\documentclass[12pt,a4paper]{amsart}
\usepackage{amsmath,amsfonts,amssymb,amscd,amsthm,calc,enumerate}
\usepackage{caption}[2006/03/21]
\DeclareCaptionFormat{nocaps}{#1 #3\par}

\newtheorem{theorem}{Theorem}[section]

\newtheorem{lemma}[theorem]{Lemma}
\newtheorem{corollary}[theorem]{Corollary}
\newtheorem{proposition}[theorem]{Proposition}

\theoremstyle{definition}
\newtheorem{definition}[theorem]{\bf Definition}

\newcommand{\concat}{\hspace*{.2mm}\widehat{\mbox{ }}\hspace*{.2mm}}

\newcommand{\set}[1]{\left\{#1\right\}}
\newcommand{\bigset}[1]{\big\{ #1 \big\}}

\renewcommand{\leq}{\leqslant}
\renewcommand{\geq}{\geqslant}
\renewcommand{\le}{\leqslant}
\renewcommand{\ge}{\geqslant}

\newcommand{\imp}{\rightarrow}

\newcommand{\M}{{\mathfrak M}}
\newcommand{\A}{\mathcal{A}}
\newcommand{\B}{\mathcal{B}}
\newcommand{\C}{\mathcal{C}}

\newcommand{\X}{\mathcal{X}}
\newcommand{\Y}{\mathcal{Y}}
\newcommand{\Z}{\mathcal{Z}}
\newcommand{\bfS}{\mathbf{S}}
\newcommand{\IPC}{{\sf IPC}}
\newcommand{\Jan}{{\sf Jan}}

\newcommand{\Th}{{\rm Th}}
\newcommand{\join}{+}
\newcommand{\meet}{\times}
\newcommand{\0}{0}

\renewcommand{\Join}{{\textstyle \sum}}
\newcommand{\Meet}{{\textstyle \prod}}
\newcommand{\dual}[1]{#1^{\rm op}}
\renewcommand{\phi}{\varphi}
\renewcommand{\bigvee}{\sum }

\begin{document}

\title[Intuitionistic logic and Muchnik degrees]
{Intuitionistic logic and Muchnik degrees}

\author[A. Sorbi]{Andrea Sorbi}
\address[Andrea Sorbi]{University of Siena \\
Dipartimento di Scienze Matematiche ed Informatiche ``Roberto Magari''\\
Pian dei Mantellini 44, 53100 Siena, Italy. } \email{sorbi@unisi.it}
\author[S. A. Terwijn]{Sebastiaan A. Terwijn}
\address[Sebastiaan A. Terwijn]{Radboud University Nijmegen\\
Department of Mathematics\\
P.O. Box 9010, 6500 GL Nijmegen, the Netherlands.
} \email{terwijn@math.ru.nl}

\begin{abstract}
We prove that there is a factor of the Muchnik lattice that captures intuitionistic
propositional logic. This complements a now classic result of Skvortsova for the Medvedev
lattice.
\end{abstract}


\date{}

\maketitle

\section{Introduction}
Amongst the structures arising from computability theory, the Medvedev and the Muchnik
lattices stand out for several distinguished features and a broad range of applications.
In particular these lattices have additional structure that makes them suitable as models
of certain propositional calculi. The structure of the Medvedev lattice as a Brouwer
algebra, and thus as a model for propositional logics, has been extensively studied in
several papers, see e.g.~\cite{Medvedev},~\cite{Skvortsova},
\cite{Sorbi:Brouwer},~\cite{Sorbi-Terwijn:Intermediate},~\cite{Terwijn:Constructive} .
Originally motivated, \cite{Medvedev}, as a formalization of Kolmogorov's calculus of
problems \cite{Kolmogorov:Problems}, the Medvedev lattice fails to provide an exact
interpretation of the intuitionistic propositional calculus $\IPC$;
however,~\cite{Skvortsova},  there are initial segments of the Medvedev lattice that
model exactly $\IPC$. On the other hand, little is known about the structure of the
Muchnik lattice, and of its dual, as Brouwer algebras. The goal of this paper is to show
that there are initial segments of the Muchnik lattice, in which the set of valid
propositional sentences coincides with $\IPC$. From this, it readily follows that the
valid propositional sentences that are valid in the Muchnik lattice are exactly the
sentences of the so-called logic of the weak law of excluded middle
(\cite{Sorbi:Brouwer}).
Similar results (as announced, with outlined proofs, in \cite{Sorbi:Quotient}) hold of
the dual of the Muchnik lattice: detailed proofs are provided in Section \ref{sct:dual}.

For all unexplained notions from computability theory, the reader is referred to Rogers
\cite{Rogers:Book}; our main source for Brouwer algebras and the algebraic semantics of
propositional calculi is Rasiowa-Sikorski \cite{Rasiowa-Sikorski:Book}. A comprehensive
survey on the Medvedev and Muchnik lattices, and their mutual relationships, can be found
in \cite{Sorbi:Medvedev-Survey}. Throughout the paper we use the symbols $\join$ and
$\meet$ to denote the join and meet operations, respectively, in any lattice.

\subsection{The Medvedev and the Muchnik lattices}
Although our main object of study is the Muchnik lattice, reference to the Medvedev
lattice will be sometimes useful. Therefore, we start by reviewing some basic definitions
and facts concerning both lattices. Following Medvedev \cite{Medvedev}, a \emph{mass
problem} is a set of functions from the set of natural numbers $\omega$, to itself. There
are two natural ways to extend Turing reducibility to mass problems: one could say,
following \cite{Medvedev}, that a mass problem $\A$ is \emph{reducible} to a mass problem
$\B$ (denoted by $\A \le \B$), if there is an oracle Turing machine by means of which
every function of $\B$, when supplied to the machine as an oracle, computes some function
of $\A$. (Any oracle Turing machine defines in this sense a partial mapping from
$\omega^\omega$ to $\omega^\omega$, called a \emph{partial computable functional}.) A
different approach, which consists in dropping uniformity, leads to \emph{Muchnik
reducibility}, \cite{Muchnik:Lattice}, denoted by $\le_w$: here $\A \le_w \B$, if for
every $g \in \B$ there is an oracle Turing machine which computes some $f \in \A$ when
given $g$ as an oracle. This amounts to saying that $\A \le_w \B$ if and only if for
every $g \in \B$ there is some $f \in \A$ such that $f \le_T g$. Both definitions may be
viewed as attempts at formalizing Kolmogorov's idea of a calculus of problems:
Kolmogorov's informal problems are now identified with mass problems; to ``solve'' a mass
problem means to find a computable member in it; $\A \le \B$ and $\A \le_w \B$ are then
formalizations of ``$\A$ is less difficult than $\B$'', as one can solve $\A$ given any
solution to $\B$. In the same vein, one can introduce a formal ``calculus'' of mass
problems, by defining $\mathcal{A} \join \mathcal{B}=\left\{ f \oplus g: f \in
\mathcal{A} \text{ and } g \in \mathcal{B} \right\}$, where
$$
f \oplus g(x)=\left\{
\begin{array}{ll}
f(y), & \hbox{if $x=2y$,} \\
g(y), & \hbox{if $x=2y+1$;}%
\end{array}
\right.
$$
and $\mathcal{A} \meet \mathcal{B}=\langle 0 \rangle \concat \mathcal{A} \cup \langle 1
\rangle\concat \mathcal{B}$, where in general, for $i \in \omega$ and a given mass
problem $\C$, $\langle i \rangle \concat \mathcal{C}= \left\{ \langle i \rangle \concat
f: f \in \mathcal{C} \right\}$, and $\langle i \rangle \concat f$ denotes the
concatenation of the string $\langle i\rangle$ with the function $f$. We see that $\A
\join \B$ has a solution if and only if both $\A$ and $\B$ have solutions; and $\A \meet
\B$ has a solutions if and only if at least one of them has. Being preordering relations,
both $\le$ and $\le_w$ give rise to degree structures: the equivalence class $\deg_M(\A)$
of a mass problem $\A$, under the equivalence relation $\equiv$ generated by $\le$, is
called the \emph{Medvedev degree} of $\A$; the equivalence class $\deg_w(\A)$ of a mass
problem $\A$, under the equivalence relation $\equiv_w$ generated by $\le_w$ is called
the \emph{Muchnik degree} of $\A$. The corresponding degree structures are not only
partial orders, but in fact bounded distributive lattices, with operations of join and
meet (still denoted by $+$ and $\times$) defined through the corresponding operations on
mass problems. It is easily seen that both lattices are distributive. The lattice of
Medvedev degrees is called the \emph{Medvedev lattice}, denoted by $\mathfrak{M}$; the
lattice of Muchnik degrees is called the \emph{Muchnik lattice}, denoted by
$\mathfrak{M}_w$. Finally the least element in both lattices is the degree of any mass
problem containing some computable function; and the greatest element is the degree of
the mass problem $\emptyset$.

A \emph{Muchnik mass problem} $\A$ is a mass problem satisfying: $f\in \A \text{ and }f
\le_T g \Rightarrow g \in \A$.
\begin{lemma}\label{lem:useful}
The following hold:
\begin{enumerate}
  \item\label{item:closure} for every mass problem $\A$, there is a unique Muchnik
      mass problem $\A^w$ such that $\A\equiv_w \A^w$:
  \item\label{item:un-int} $\M_w$ is a completely distributive complete lattice, with
      $\A \meet \B \equiv_w \A \cup \B$, and if $\A$ and $\B$ are Muchnik mass
      problems then $\A \join \B \equiv_w \A \cap \B$.
  \end{enumerate}
\end{lemma}

\begin{proof}
Define $\A^w=\left\{f: (\exists g \in \A)[g \le_T f]\right\}$. $\M_w$ is complete: if
$\set{\A_i: i \in I}$ is any collection of mass problems, then the infimum and the
supremum of the corresponding Muchnik degrees are given by
\begin{align*}
\Meet \set{\deg_w(\A_i): i \in I} &=\deg_w(\bigcup \set{\A_i: i \in I}),\\
\Join \set{\deg_w(\A_i): i \in I} &=\deg_w(\bigcap \set{\A_i^w: i \in I}).
\end{align*}
We will often extend the $\Meet$ and $\Join$ operations to mass problems by defining:
\begin{align*}
  \Meet \set{\A_i: i \in I}&=\bigcup \set{\A_i: i \in I} \\
  \Join \set{\A_i: i \in I}&=\bigcap \set{\A_i^w: i \in I}.
\end{align*}
Complete distributivity follows from the fact that infima and suprema are essentially
given by set theoretic unions and intersections.
\end{proof}

Both in $\mathfrak{M}$ and in $\mathfrak{M}_w$, a degree $\bfS$ is called a \emph{degree
of solvability} if it contains a singleton. The following considerations concerning
degrees of solvability apply to both $\M$ and $\M_w$: it is easy to see that the degrees
of solvability form an upper semilattice, with least element, which is isomorphic to the
upper semilattice, with least element, of the Turing degrees; for every degree of
solvability $\bfS$ there is a unique minimal degree $>\bfS$ that is denoted by $\bfS'$
(cf.\ Medvedev \cite{Medvedev}). If $\bfS = \deg_M(\{f\})$ then $\bfS'$ is the degree of
the mass problem
$$
\{f\}' = \bigset{\langle n \rangle \concat g: f <_T g \wedge \Phi_n(g) = f },
$$
where $\{\Phi_n\}_{n \in \omega}$ is an effective list of all partial computable
functionals; note further that for any $f$ we have $\{f\}' \equiv_w \{g\in \omega^\omega:
f<_T g \}$ so that in $\M_w$ we can use this simplified version of $\{f\}'$. In
particular, $\0' = \bigset{g: g>_T 0}$ is the minimal nonzero Muchnik degree.

\section{Brouwer algebras and intermediate propositional calculi} We now recall the
basic definitions and facts about Brouwer and Heyting algebras, and their applications to
propositional logics.

\begin{definition}\label{defn:Brouwer}
A distributive lattice $L$ with least and largest elements $0$ and $1$, and with
operations of join and sup denoted by $\join$ and $\meet$, respectively, is a {\em
Brouwer algebra\/} if for every pair of elements $a$ and $b$ there is a smallest element,
denoted by $a\imp b$, such that $a\join (a\imp b) \geq b$. Thus a Brouwer algebra can be
viewed as an algebraic structure with three binary operations $\join, \meet, \imp$,
together with the nullary operations $0,1$. For applications to propositional logic, it
is also convenient to enrich the signature of a Brouwer algebra with a further unary
operation $\neg$, given by $\neg a=a \imp 1$.
\end{definition}

Given a Brouwer algebra $L$, we can identify a propositional formula $\phi$, having $n$
variables, with an $n$-ary polynomial $p_\phi$ of $L$, in the restricted signature
$\langle \join, \meet, \imp, \neg\rangle$: the identification makes the propositional
connectives $\lor, \wedge, \imp, \neg$ correspond to the operations $\meet, \join, \imp,
\neg$ of $L$, respectively. (For polynomials in the sense of universal algebra, see for
instance \cite{Graetzer:Universal}.) The polynomial $p_\phi$ can in turn be considered as
a function $p_\phi:L^n \longrightarrow L$.

\begin{definition}\label{defn:evaluating} Let $L$ be a Brouwer algebra. A propositional
formula $\phi$ having $n$ variables is {\em true\/} in $L$ if $p_\phi(a_0, \ldots,
a_{n-1})=0$ for all $(a_0, \ldots, a_{n-1})\in L^n$. The set of all formulas that are
true in $L$ is denoted by $\Th(L)$.
\end{definition}

The dual notion is studied as well.

\begin{definition} \label{dual}
A distributive lattice $L$ with least and largest elements $0$ and $1$ is a {\em Heyting
algebra\/} if its dual $\dual{L}$ is a Brouwer algebra. That is, $a\imp b$ is the largest
element of $L$ such that $a \meet (a\imp b) \leq b$. A propositional formula is {\em
true\/} in the Heyting algebra $L$ if the polynomial $\dual{p}_\phi$, obtained from
$p_\phi$ by interchanging $\meet$ and $\join$, evaluates to $1$ under every valuation of
its variables with elements from $L$. The set of all formulas that are true in $L$ as a
Heyting algebra is denoted by $\Th_H(L)$. Note that $\Th_H(L) = \Th(\dual{L})$.
\end{definition}

\begin{lemma} \label{surj}
Suppose that $L_0$ and $L_1$ are Brouwer algebras, and suppose that $F: L_0
\longrightarrow L_1$ is a Brouwer homomorphism (i.e. a homomorphism of bounded lattices,
which also preserves $\rightarrow$).
\begin{enumerate}
\item If $F$ is injective then $\Th(L_1) \subseteq \Th(L_0)$.
\item If $F$ is surjective then $\Th(L_0) \subseteq \Th(L_1)$.
\end{enumerate}
\end{lemma}
\begin{proof}
See \cite{Rasiowa-Sikorski:Book}.
\end{proof}

Given $a<b$ in a Brouwer algebra $L$, $L[a,b]$ denotes the interval
$$
[a,b]=\set{x \in  L: a \le x \le b}.
$$
We abbreviate $L[0,b]$ with $L(\le b)$, and we abbreviate $L[a,1]$ with $L(\ge a)$.

\begin{lemma}\label{lem:intervals}
Suppose that $L$ is a Brouwer algebra, and let $a,b\in L$ be such that $a<b$. Then
$L[a,b]$ is again a Brouwer algebra.
\end{lemma}
\begin{proof}
Let $\rightarrow$ be the arrow operation in $L$. Then the arrow operation $\imp_{[a,b]}$
in $L[a,b])$ is given by
$$
x \imp_{[a,b]} y= x \join (x \imp y).
$$
\end{proof}

\begin{lemma} \label{hom}
Let $L$ be a Brouwer algebra and let $a,b,c\in L$ such that $c \join a = b$. Then the
mapping $f(x) = x \join a$ is a Brouwer homomorphism of $L(\leq c)$ onto $L[a,b]$. As a
consequence, $\Th(L(\leq c)) \subseteq \Th(L[a,b])$.
\end{lemma}
\begin{proof}
See \cite[Lemma~4]{Skvortsova}.
\end{proof}

\begin{lemma} \label{surjectivity}
Let $L$ be a distributive lattice, and suppose that $x \le y$ and $z$ is arbitrary. Then
the mapping $c \mapsto c \meet z$ is a surjective homomorphism from the interval $[x,y]$
onto the interval $[x\meet z, y\meet z]$.
\end{lemma}

\begin{proof}
It is obvious that the mapping is a lattice-theoretic homomorphism. Surjectivity follows
from the fact that if $x \meet z \le u \le y \meet z$ then $u$ is the image of $x\join
(u\meet y)$.
\end{proof}

\subsection{The Medvedev and the Muchnik lattices as Brouwer algebras}

Examples of Brouwer algebras are provided by $\M$ (Medvedev \cite{Medvedev}), $\M_w$
(Muchnik \cite{Muchnik:Lattice}), and the dual $\dual{\M_w}$ (\cite{Sorbi:Someremarks}):

\begin{lemma}\label{lem:both}
The Muchnik lattice $\mathfrak{M}_w$ is both a Brouwer algebra and a Heyting algebra. The
Medvedev lattice $\mathfrak{M}$ is a Brouwer algebra.
\end{lemma}
\begin{proof}
$\mathfrak{M}_w$ is a Brouwer algebra (\cite{Muchnik:Lattice}), and a Heyting algebra
(\cite{Sorbi:Someremarks}) since it is a completely distributive complete lattice. To
show that $\mathfrak{M}_w$ is a Brouwer algebra, take for instance, on mass problems,
$$
\A \rightarrow \B=\Meet \{\C: \B \le \A \join \C\}.
$$
To show that $\mathfrak{M}$ is a Brouwer algebra (\cite{Medvedev}), on mass problems $\A,
\B$, define
$$
\A \rightarrow \B=\left\{\langle n \rangle \concat f: (\forall g\in \A)[\Phi_n(g\oplus f)\in \B]
\right\}:
$$
it is immediate that $\B \le \A \join (\A \rightarrow \B)$, and
$$
\B \le \A \join \C \Leftrightarrow \A \rightarrow \B \le \C.
$$
Since Muchnik reducibility is a nonuniform version of
Medvedev reducibility, we can also notice that
for the $\rightarrow$ operation in the Muchnik lattice as a Brouwer algebra, one can take
$$
\A \rightarrow \B=\set{f: (\forall g \in \A)(\exists h \in \B)[h \le_T g \oplus f]}.
$$
In terms of the calculus of problems, we observe that with these definitions of
$\rightarrow$, for both Medvedev and Muchnik reducibility one has that $\A \rightarrow
\B$ is a mass problem such that any solution to it, together with any solution to $\A$,
gives a solution to $\B$.
\end{proof}

For either $\M$ or $\M_w$, Definition \ref{defn:evaluating} amounts to saying that a
propositional sentence is valid if and only if every substitutions of mass problems to
the propositional variables in the sentence yields a solvable problem. Let $\IPC$ denote
the intuitionistic propositional calculus (see \cite{Rasiowa-Sikorski:Book} for a
suitable definition of axioms and rules of inference), and let $\Jan$ be the intermediate
propositional logic obtained by adding to $\IPC$ the so called \emph{weak law of excluded
middle}, i.e. the axiom scheme $\neg \alpha \lor \neg \neg \alpha$, where $\alpha$ is any
propositional sentence. It is known (Medvedev \cite{Medvedev:finite}, Jankov
\cite{Jankov:Weak}, Sorbi \cite{Sorbi:Brouwer}) that $\Th(\M)=\Jan$. Also,
$\Th(\M_w)=\Jan$ (announced in \cite{Sorbi:Brouwer}).

By lattice theory, if $L$ is a Brouwer algebra, then the Brouwer algebra $L(\le b)$ is
lattice isomorphic to the quotient lattice obtained by dividing $L$ modulo the principal
filter generated by $b$; likewise, $L(\ge a)$ is isomorphic to the quotient lattice
obtained by dividing $L$ modulo the principal ideal generated by $a$. The difference
between these two quotients, see e.g. \cite{Rasiowa-Sikorski:Book}, is that congruences
given by ideals are also congruences of Brouwer algebras, and thus there is a surjective
Brouwer homomorphism from $L$ into $L(\ge a)$, giving $\Th(L) \subseteq \Th(L(\ge a))$ by
Lemma \ref{surj}. In order to find exact interpretations of $\IPC$ in terms of mass
problems, one should then turn attention to initial segments of the Medvedev lattice,
i.e. to Brouwer algebras of the form  $\M(\le \mathbf{A})$, where $\mathbf{A}$ is a
nonzero Medvedev degree.

\begin{theorem} {\rm (Skvortsova \cite{Skvortsova})} \label{Skvortsova}
There exists $\mathbf{A}$ such that $\Th(\M(\le \mathbf{A}))=\IPC$.
\end{theorem}

It is still an open problem (raised by Skvortsova \cite[p.134]{Skvortsova}) whether there
is a Medvedev degree $\mathbf{A}$ that is the infimum of finitely many Muchnik degrees
(i.e. Medvedev degrees containing Muchnik mass problems) such that $\Th(\M(\le
\mathbf{A}))$ coincides with $\IPC$. The paper \cite{Sorbi-Terwijn:Intermediate} is
dedicated to initial segments of the Medvedev lattice and their theories as intermediate
propositional logics.

\section{Capturing $\IPC$ with Brouwer and Heyting algebras}\label{sec:algebras}

Consider the following classic result about $\IPC$ due to McKinsey and Tarski, that
provides an algebraic semantics for $\IPC$ using Brouwer algebras. (The result also
follows from the results in Ja\'{s}kowski \cite{Jaskowski}).

\begin{theorem} {\rm (Ja\'{s}kowski \cite{Jaskowski},
McKinsey and Tarski \cite{McKinseyTarski})}  \label{JMcT}
\begin{align*}
  \IPC &= \bigcap\bigset{\Th(B) : B \text{ a finite Brouwer algebra}}.\\
   &= \bigcap\bigset{\Th_H(H) : H \text{ a finite Heyting algebra}}.
\end{align*}
\end{theorem}

We wish to narrow down the family of Brouwer algebras and Heyting algebras needed for
this result, in order to suit our needs in the next section. The result we will need
later is formulated below as Corollary~\ref{corBrouwer}.

For a given lattice $L$, let $J(L)$ denote the partial order of nonzero join-irreducible
elements of~$L$. Recall the well known duality between finite posets and finite
distributive lattices. Obviously, for every finite distributive lattice $L$, $J(L)$ is a
poset, and conversely, for every finite poset $P$ we obtain a finite distributive lattice
$H(P)$ by considering the downwards closed subsets of $P$ (\cite[Theorem
II.1.9]{Graetzer}). These operations are inverses of each other, as $H(J(L))\simeq L$ (as
lattices), and $J(H(P))\simeq P$ (as posets).

The following is a useful notion from the theory of categories. An equational category is
a category whose objects form a variety of algebras, and the morphism are just the
homomorphisms.

\begin{definition}
An object of an equational category $L$ is \emph{weakly projective} if for every onto
morphism $f: L_0 \twoheadrightarrow L_1$ and every morphism $g: L \rightarrow L_1$, there
exists a morphism $h: L \rightarrow L_0$ such that $g=f \circ h$.
\end{definition}

\begin{lemma}\label{lem:weak-dual}
A distributive lattice $L$ is weakly projective
if and only if its dual $\dual{L}$ is.
\end{lemma}

\begin{proof}
It can be shown, see e.g. \cite[Theorem~1.14]{BalbesDwinger}, that in a nontrivial
equational category, an object is weakly projective if and only if it is a retract of a
free algebra. (Recall that $A$ is a retract of $B$, if there are morphisms $f:A
\rightarrow B, g:B \rightarrow A$ such that $g\circ f=1_A$.) If $L$ is weakly projective,
and $L$ is a retract of a free distributive lattice $F$, then $\dual{L}$ is a retract of
$\dual{F}$ which is still free.
\end{proof}

When considering the category of distributive lattices, the following characterization of
the finite weakly projective objects is available:

\begin{theorem}\cite[Corollary V.10.9]{BalbesDwinger}
A finite distributive lattice $L$ is weakly projective if and only if whenever $a$ and
$b$ are join-irreducible in $L$ also $a\meet b$ is join-irreducible.
\end{theorem}

The following property from \cite{Terwijnta} gives an alternative characterization of
finite weakly projective distributive lattices:

\begin{definition}\label{ddlike}
A finite distributive lattice $L$ is {\em double diamond-like\/} (\emph{dd-like}, for
short) if in the poset $J(L)$ there are two incomparable elements with at least two
minimal upper bounds.
\end{definition}

\begin{proposition}
A finite distributive lattice $L$ is weakly projective if and only if it is not dd-like.
\end{proposition}
\begin{proof}
When $L$ is weakly projective then every pair $a,b$ of join-irreducible elements has a
greatest lower bound $a\meet b$ that is join-irreducible, and hence $a\meet b$ is also
the greatest lower bound of $a$ and $b$ in the poset $J(L)\cup\{0\}$. Hence $L$ is not
dd-like.

Conversely, if $L$ is not weakly projective then there are $a, b \in J(L)$ such that
$a\meet b$ is join-reducible. Since any element in a finite distributive lattice can be
written as a finite join of join-irreducible elements, there is a finite set $X\subseteq
J(L)$ such that $a\meet b = \bigvee X$. Since $a\meet b$ itself is join-reducible, there
are at least two maximal elements $x,y \in X$. Then both $a$ and $b$ are maximal lower
bounds of $x$ and $y$ in $J(L)$, hence $L$ is dd-like.
\end{proof}

We now undertake the task of characterizing $\IPC$ by  suitably restricted families of
Heyting algebras and Brouwer algebras. We can in fact start from a family that was
already used by Ja\'{s}kowski, by observing that it has certain additional properties.
The result we will need later is formulated below as Corollary~\ref{corHeyting}.

\begin{lemma} \label{preservation}
If $A$ and $B$ are finite distributive lattices that are not dd-like, then also $A\times
B$ is not dd-like.
\end{lemma}
\begin{proof}
We need in fact that only one of $A$ and $B$ is not dd-like. Suppose that $A$ is not
dd-like. Note that $(a,b) \in A\times B$ is join-irreducible if and only if $a\in J(A)$
and $b\in J(B)$. Suppose that $A\times B$ is not dd-like, say $J(A\times B)$ contains the
following configuration:
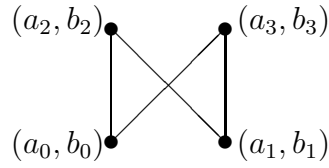
\begin{center}
\setlength{\unitlength}{1.5mm}
\begin{picture}(10,13)

\put(0,0){\circle*{1.2}} \put(0,10){\circle*{1.2}} \put(10,0){\circle*{1.2}}
\put(10,10){\circle*{1.2}}

\put(0,0){\line(0,1){10}} \put(10,0){\line(0,1){10}} \put(0,0){\line(1,1){10}}
\put(0,10){\line(1,-1){10}}

\put(-8.8,-1){$(a_0,b_0)$} \put(-8.8,10){$(a_2,b_2)$} \put(11,-1){$(a_1,b_1)$}
\put(11,10){$(a_3,b_3)$}

\end{picture}

\end{center}

\smallskip\noindent
Here the pairs $(a_2,b_2)$ and $(a_3,b_3)$ are minimal upper bounds for $(a_0,b_0)$ and
$(a_1,b_1)$ in $J(A\times B)$. Then in $J(A)$ the elements $a_2$ and $a_3$ are upper
bounds for $a_0$ and $a_1$. Since by assumption $A$ is not dd-like, not both of $a_2$ and
$a_3$ are minimal upper bounds. Say $a_2$ is not minimal, and that $a_0, a_1 \leq a <
a_2$ in $J(A)$. Replacing $(a_2,b_2)$ by $(a,b_2)$, we see that $(a_2, b_2)$ was not a
minimal upper bound of $(a_0,b_0)$ and $(a_1,b_1)$, contrary to assumption.
\end{proof}

\noindent We use the following result of Ja\'skowski~\cite{Jaskowski}, (cited in
Szatkowski~\cite[p41]{Szatkowski}). Given two Heyting algebras $A$ and $B$, let $A + B$
be the algebra obtained by stacking $B$ on top of $A$, identifying $0_B$ with $1_A$.
(This notion of sum is from Troelstra~\cite{Troelstra}.) Given $A$ and $B$, the Cartesian
product $A\times B$ is again a Heyting algebra. Let $A^n$ denote the $n$-fold product of
$A$.

Inductively define the following sequence of Heyting algebras. Let $I_1$ be the
two-element Boolean algebra, and let
$$
I_{n+1} =  I^n_n + I_1.
$$

The following theorem characterizes $\IPC$ in terms of Heyting algebras:

\begin{theorem}\label{thm:jaskowski} {\rm (Ja\'skowski \cite{Jaskowski})}
$\IPC = \bigcap_n \Th_H(I_n)$.
\end{theorem}

\begin{corollary} \label{corHeyting}
There is a collection $\{H_n\}_{n\in\omega}$ of finite Heyting algebras such that
$$
\IPC = \bigcap_n \Th_H(H_n),
$$
and such that for every $n$, $H_n$ is weakly projective.
\end{corollary}
\begin{proof}
Note that the lattices $I_n$ defined above are all distributive lattices, and because
they are finite they are automatically Heyting algebras. We claim that every $I_n$ is not
dd-like. This is clearly true for $n=1$. Suppose that $I_n$ is not dd-like. Then by
Lemma~\ref{preservation} also $I^n_n$ is not dd-like. It follows immediately that
$I_{n+1} = I^n_n + I_1$ is also not dd-like. Hence all $I_n$ are finite Heyting algebras
that are not dd-like, and hence we can simply take $H_n = I_n$.
\end{proof}

\begin{corollary} \label{corBrouwer1}
There is a collection $\{B_n\}_{n\in\omega}$ of finite Brouwer algebras such that
$$
\IPC = \bigcap_n \Th(B_n),
$$
and such that for every $n$, $B_n$ is weakly projective.
\end{corollary}
\begin{proof}
Consider any propositional formula $\phi \notin \IPC$. Then by Theorem
\ref{thm:jaskowski} there exists a weakly projective finite lattice $H_n$ and an
evaluation of $\dual{p}_\phi$ for which $\dual{p}_\phi\ne 1$, and thus, for this
evaluation in $\dual{H}_n$, $p_\phi\ne 0$,
showing that $\phi \notin \Th(\dual{H}_n)$. It remains to show that $B_n=\dual{H}_n$ is
weakly projective: this follows from Lemma \ref{lem:weak-dual}.
\end{proof}

Summarizing, we have:

\begin{corollary} \label{corBrouwer}
We have
\begin{align*}
\IPC &= \bigcap\bigset{\Th(B) : B \text{  finite and weakly projective}}.\\
   &= \bigcap\bigset{\Th_H(H) : H \text{  finite and weakly projective}}.
\end{align*}
\end{corollary}

\section{A factor of the Muchnik lattice that captures $\IPC$}
\label{factorMw}

In this section we prove that there is a factor of $\M_w$, obtained by dividing $\M_w$
with a principal filter, that has $\IPC$ as its theory. Hence we see that the analogue of
Skvortsova's result (Theorem~\ref{Skvortsova}) holds for $\M_w$. We will be very liberal
with notation, frequently confusing Muchnik degrees with their representatives.

The property of dd-like lattices (Definition~\ref{ddlike}) was used to characterize the
lattices that are isomorphic to an interval of $\M_w$:

\begin{theorem} {\rm (Terwijn \cite{Terwijnta})} \label{characterization}
For any finite distributive lattice $L$ the following are equivalent:
\begin{enumerate}[\rm (i)]

\item $L$ is isomorphic to an interval in $\M_w$,

\item $L$ is not double diamond-like,

\item $L$ does not have a double diamond-like lattice as a subinterval.

\end{enumerate}
\end{theorem}

\noindent Let $\{B_n\}_{n \in \omega}$ be the family of Brouwer algebras from
Corollary~\ref{corBrouwer}. Since $B_n$ is not dd-like, by Theorem~\ref{characterization}
there are sets $\X_n$ and $\Y_n$ such that the interval $[\X_n, \Y_n]$ in $\M_w$ is
isomorphic to $B_n$ for every~$n$. This is an isomorphism of finite distributive
lattices, hence it is automatically an isomorphism of Brouwer algebras.

It is useful to remind the reader of some of the details of the construction in
\cite{Terwijnta}. Let $J_n=J(B_n)$ be the set of the nonzero join-irreducible elements of
$B_n$; since $B_n$ is not dd-like, $J_n$ is an initial segment of an upper semilattice.
Embed $J_n$ as an interval of the Turing degrees (this can be done, by a classical result
of Lachlan and Lebeuf \cite{Lachlan-Lebeuf}, stating that for every Turing degree
$\mathbf{a}$, every countable upper semilattice with least element $0$ is isomorphic to
an interval of the Turing degrees with bottom $\mathbf{a}$). For every Turing degree in
the range of this embedding, choose a representative, as a function $f \in
\omega^\omega$, and for convenience, let us identify $J_n$ with the set of these chosen
representatives. For every $A \subseteq J_n$, let $\hat A$ denote the elements of $A$
that are $\le_T$-maximal, i.e. maximal with respect to Turing reducibility.

Inspection of the proof of Theorem~3.11 in \cite{Terwijnta} shows that there is a set
$\Z_n$ such that {\setlength\arraycolsep{2pt}
\begin{eqnarray}
\X_n &=& \Z_n \meet J_n  \nonumber \\
\Y_n &=& \Z_n \meet \Meet \bigset{ \{f\}' : f \in \hat J_n} \label{Yn}
\end{eqnarray}
and $B_n$ is isomorphic to the interval $[\X_n, \Y_n]$ of the Muchnik lattice.
Furthermore, we have that $\Z_n=\bigcup_{f \in J_n}\Z^f_n$, where
\begin{equation} \label{Zn}
\Z^f_n=\bigset{g \in \{f\}': g|_T h \text{ for all covers $h$ of $f$ in $J_n$}}.
\end{equation}
The sets $J_n$ come from embedding results into the Turing degrees, and we have rather
great freedom in picking them. In particular, we may pick them such that they satisfy for
every $n\neq m$,
\begin{equation} \label{eigenschap}
(\forall f\in J_n) \;[\{f\}\not\geq_w \Z_m]
\end{equation}
and
\begin{equation} \label{incomp}
\left.
\begin{array}{r}
f \in \hat J_m \\
g \in \hat J_n \\
h \in J_n
\end{array} \right\} \Longrightarrow
f \oplus h >_T g.
\end{equation}
To obtain this, it is enough to embed as an interval of the Turing degrees, the
upper semilattice $J$ defined as follows: First, let
$$
U=\bigcup_n \{n\}\times J_n
$$
(where, again, $J_n=J(B_n)$) and in $U$ define $(n,x) \le (m, y)$ if and only if $n=m$
and, in $J_n$, $x \le y$; finally define $J$ by adding a least element and a greatest
element to $U$. Clearly $J$ is a countable upper semilattice with least element, and thus
can be embedded as an interval of the Turing degrees: under this embedding each $J_n$ is
embedded as an interval of the Turing degrees, with the desired properties.

Define {\setlength\arraycolsep{2pt}
\begin{eqnarray*}
\Z &=& \bigcup_{n\in\omega} \Z_n,\\
\hat{\X}_n &=& \Z \meet J_n \equiv_w \X_n \meet \Z, \\
\hat{\Y}_n &=& \Z \meet \Meet \bigset{ \{f\}' : f \in \hat J_n}
\equiv_w \Y_n \meet \Z.
\end{eqnarray*}
}%

\begin{lemma}\label{lem:interval} The interval $[\X_n,\Y_n]$ is isomorphic to the interval
$[\hat{\X}_n,\hat{\Y}_n]$.
\end{lemma}
\begin{proof}
Define a mapping from $[\X_n,\Y_n]$ to $[\X_n\meet\Z,\Y_n\meet\Z]$ by $\C \mapsto
\C\meet\Z$. Clearly the mapping is a homomorphism, and it is surjective by
Lemma~\ref{surjectivity}. We check that it is also injective: Suppose that $\C_0$,
$\C_1\in [\X_n,\Y_n]$ and that $\C_0 \meet \Z \equiv_w \C_1 \meet \Z$. We claim that
$\C_0 \geq_w \C_1 \meet \Z_n$: Suppose that $g\in\C_0$. Then $\{g\} \geq_w \X_n = \Z_n
\meet J_n$. If $\{g\}\geq_w \Z_n$ then clearly it can be mapped to $\C_1\meet\Z_n$. If
$\{g\}\not\geq_w \Z_n$ then we have $\{g\}\geq_w J_n$, and it follows from \eqref{Zn} and
the fact that $J_n$ is an initial segment that $g\in J_n$. But in this case it follows
from \eqref{eigenschap} and the assumption $\C_0 \geq_w \C_1 \meet \Z$ that $\{g\}\geq_w
\C_1 \meet \Z_n$. Hence $\C_0 \geq_w \C_1 \meet \Z_n \equiv_w \C_1$ (note that $\Z_n
\geq_w \C_1$ since $\Y_n\geq_w \C_1$), and symmetrically we have that $\C_1\geq_w \C_0$,
hence $\C_0\equiv_w \C_1$. 
\end{proof}

Now let
$$
\hat{\Y} = \bigcup_{n\in\omega} \hat{\Y}_n.
$$

\begin{lemma} \label{lemmajoin}
$\hat{\Y}\join \hat{\X}_n \equiv_w \hat{\Y}_n$ for every~$n$.
\end{lemma}
\begin{proof}
The direction $\leq_w$ is immediate from $\hat{\Y}\leq_w \hat{\Y}_n$ and
$\hat{\X}_n\leq_w \hat{\Y}_n$. For the other direction, suppose that $g\in\hat{\Y}$ and
$h\in\hat{\X}_n$. We have to show that $g\oplus h$ computes some function in
$\hat{\Y}_n$. Suppose that $g\in \hat{\Y}_m$. If $n=m$ then we are done.
If either $g$ or $h$ is in $\langle 0 \rangle\concat\Z$ then we are also done because
$\langle 0 \rangle\concat\Z \subseteq \hat{\Y}_n$.

In the remaining case we have $n\neq m$, $h\in \langle 1 \rangle\concat J_n$, and $g\in
\{f\}'$ for some $f\in \hat J_m$. Let $l$ be any element of $\hat J_n$. Then by
\eqref{incomp} we have $f\oplus h >_T l$, hence $g\oplus h \geq_T f\oplus h \in\{l\}'
\geq_w \hat{\Y}_n$.
\end{proof}

\begin{theorem} \label{thm:firstgood}
There exists a set of reals $\hat{\Y}$ such that $\Th(\M_w(\le_w \hat{\Y})) = \IPC$.
\end{theorem}
\begin{proof}
Let $\hat{\X}_n$, $\hat{\Y}_n$, and $\hat{\Y}$ be as above. Since by
Lemma~\ref{lemmajoin} we have $\hat{\Y}\join \hat{\X}_n \equiv_w \hat{\Y}_n$ for
every~$n$, by Lemma~\ref{hom} we have that
$$
\Th\big(\M_w(\leq_w \hat{\Y})\big) \subseteq
\bigcap_n \Th\big([\hat{\X}_n,\hat{\Y}_n]\big)
= \bigcap_n \Th(B_n) = \IPC.
$$
The equality $\Th\big(\M_w(\leq_w \hat{\Y})\big)= \IPC$ follows since $\IPC \subseteq
\Th\big(\M_w(\leq_w \hat{\Y})\big)$ holds for any~$\hat{\Y}$.
\end{proof}

\section{$\M_w$ as a Heyting algebra}\label{sct:dual}

For the dual of $\M_w$ we have a similar result, but easier to prove and in fact
stronger: the result, and its consequences, listed below, were already noticed in Sorbi
\cite{Sorbi:Quotient}, with sketched proof.

Let $\{H_n\}_{n\in\omega}$ be the family of Heyting algebras from
Corollary~\ref{corHeyting}. We refer to a result from \cite{Terwijnta} (the right-to left
implication appeared also in \cite{Sorbi:Quotient}):

\begin{lemma}\label{lem:initial}
A finite distributive lattice is isomorphic to an initial segment of the Muchnik lattice
if and only if it is weakly projective, and $0$ is meet-irreducible.
\end{lemma}

\begin{corollary} \label{cor:secondgood}
$\IPC=\Th_H(\M_w(\ge \0'))$.
\end{corollary}

\begin{proof}
For every weakly projective finite lattice $H$, define $H^+=H + I_1$ (using the notation
of section~\ref{sec:algebras}.) Notice that $H$ is isomorphic to a factor of $H^+$,
obtained by dividing by the principal filter generated by $1_H$, that is the image of the
top element of $H$ into $H^+$. Since filters provide congruences of Heyting algebras, we
have by Lemma \ref{surj} that
$$
\Th_H(H^+)\subseteq \Th_H(H).
$$
It follows:
$$
\IPC=\bigcap \bigset{\Th_H(H):
H \text{ finite, weakly projective, with join-irreducible $1$}}.
$$
Suppose now that $H$ is a finite, weakly projective distributive lattice, with
join-irreducible $1$: let $H^-$ be such that $H=(H^-)^+$. Embed $I_1 + H^-$ as an initial
segment of $\M_w$, which is possible by Lemma~\ref{lem:initial}. Let $F$ be the
embedding, which is also a Heyting algebra embedding, since the range of $F$ is an
initial segment. Then the mapping
$$
G(x)=
\left\{
  \begin{array}{ll}
    F(x), & \hbox{if $x \in H^-$;} \\
   \mathbf{1}_{\M_w},
& \hbox{if $x=1_H$}
  \end{array}
\right.
$$
is a Heyting embedding of $H$ into $\M_w(\ge_w \0'))$. Thus $\IPC=\Th_H(\M_w(\ge_w \0'))$
by Lemma~\ref{surj}.
\end{proof}

A proof of the following corollary was already outlined in Sorbi~\cite{Sorbi:Quotient}.
\begin{corollary}\label{cor:full}
 $\Th_H(\M_w)=\Th(\M_w)=\Jan$.
\end{corollary}

\begin{proof}
Let us show that $\Th_H(\M_w)=\Jan$. For every Heyting algebra $H$ let $H_+=I_1 + H$. Let
us say that a propositional formula is \emph{positive} if it does not contain the
connective $\neg$, and, for every Heyting algebra $H$, let $\Th_H^{\rm
pos}(H)=\bigset{\phi\in \Th_H(H): \text{ $\phi$ positive}}$. We claim that $\Th_{H}^{\rm
pos}(H_+)\subseteq \Th_H^{\rm pos}(H)$: for this, one can show by induction on the
complexity of a positive $\phi$ that for every $\overline{x} \in H^n$,
$p^{H}_\phi(\overline{x})=p^{H_+}_\phi(\overline{x})$. Notice also that for every Heyting
algebra $H$, and any propositional formula $\alpha$, we have that $\neg \alpha \vee \neg
\neg \alpha \in \Th_H(H_+)$, i.e. $\Jan\subseteq \Th_H(H_+)$. Let $H=\M_w(\ge \0'))$, so
that $H_+=\M_w$. By Corollary~\ref{cor:secondgood} we have $\IPC=\Th_H(H)$, hence
$\IPC^{\rm pos}=\Th_H^{\rm pos}(H_+)$, and $\neg \alpha \vee \neg \neg \alpha \in
\Th_H(H_+)$. Therefore one can apply a classic result due to Jankov \cite{Jankov:Weak},
stating that $\Jan$ is the $\subseteq$-largest intermediate propositional logic $I$ such
that $\IPC^{\rm pos}=I^{\rm pos}$ and $\neg \alpha \vee \neg \neg \alpha \in I$. Thus we
also obtain the converse inclusion $\Th_H(H_+)\subseteq \Jan$.

The proof that $\Th(\M_w)=\Jan$ goes like this: let $B=\M_w(\leq_w \hat{\Y})$, with
$\hat{\Y}$ as in Theorem \ref{thm:firstgood}. Dualizing the arguments which have been
used above, show that $\Th^{\rm pos}(B^+)\subseteq \Th^{\rm pos}(B)$, but then again by
Jankov \cite{Jankov:Weak}, $\Th(B^+)=\Jan$, and since $B^+$ is Brouwer embeddable into
$\M_w$ (use $G: B^+ \longrightarrow \M_w$ which extends the embedding of $B$ into $\M_w$,
by $G(1_{B^+})=\mathbf{1}_{\M_w}$) we finally get that $\Th(\M_w)\subseteq \Jan$ (by
Lemma~\ref{surj}), and thus $\Th(\M_w)= \Jan$ since $\neg \alpha \vee \neg \neg \alpha
\in \Th(\M_w)$.
\end{proof}

\bibliographystyle{plain}
\bibliography{muchnik}

\end{document}